\documentclass[11pt]{amsart}
\usepackage{url,amsmath,amssymb,amsthm,mathtools,enumitem,tikz,graphicx,booktabs,array,comment,bm}
\usetikzlibrary{calc}
\usepackage[top=30mm, bottom=35mm, left=30mm, right=30mm]{geometry}
\usepackage{hyperref}
\usepackage{xcolor}
\usepackage[capitalize,nameinlink,noabbrev,nosort]{cleveref}
\hypersetup{
	colorlinks=true,
	linkcolor=cyan,
	citecolor=cyan,
	filecolor=cyan,
	urlcolor=cyan,
}
\newtheorem{theoremcounter}{Theorem Counter}[section]
\theoremstyle{plain}
\newtheorem{theorem}[theoremcounter]{Theorem}

\newtheorem{lemma}[theoremcounter]{Lemma}
\newtheorem{corollary}[theoremcounter]{Corollary}
\theoremstyle{definition}
\newtheorem{definition}[theoremcounter]{Definition}
\theoremstyle{remark}
\newtheorem{remark}[theoremcounter]{Remark}

\AddToHook{env/theorem/begin}{\crefalias{theoremcounter}{theorem}}
\AddToHook{env/proposition/begin}{\crefalias{theoremcounter}{proposition}}
\AddToHook{env/conjecture/begin}{\crefalias{theoremcounter}{conjecture}}
\AddToHook{env/problem/begin}{\crefalias{theoremcounter}{problem}}
\AddToHook{env/lemma/begin}{\crefalias{theoremcounter}{lemma}}
\AddToHook{env/corollary/begin}{\crefalias{theoremcounter}{corollary}}
\AddToHook{env/definition/begin}{\crefalias{theoremcounter}{definition}}
\AddToHook{env/remark/begin}{\crefalias{theoremcounter}{remark}}
\AddToHook{env/example/begin}{\crefalias{theoremcounter}{example}}

\newcommand{\qq}[1]{\langle#1\rangle}
\newcommand{\h}{\mathfrak{h}}
\newcommand{\Span}{\mathrm{span}}
\newcommand{\wt}{\mathrm{wt}}
\newcommand{\BZ}{\mathrm{BZ}}
\newcommand{\SZ}{\mathrm{SZ}}
\newcommand{\ds}{\diamondsuit}
\newcommand{\bX}{\boldsymbol{X}}

\address{Graduate School of Science and Engineering, Kagoshima University, 1-21-35 Korimoto, \newline Kagoshima, Kagoshima 890-0065, Japan}
\email{hirose@sci.kagoshima-u.ac.jp}
\address{Faculty of Mathematics, Kyushu University, Motooka 744, Nishi-ku, Fukuoka, 819-0395, Japan}
\email{nozaki.takumi.912@s.kyushu-u.ac.jp}
\address{Department of Mathematics, The University of Tokyo, 3-8-1 Komaba, Meguro-ku, Tokyo 153-8914, Japan}
\email{taiki-watanabe@g.ecc.u-tokyo.ac.jp}


\title{A unified proof of conjectures on the spaces of multiple $q$-zeta values}
\author{Minoru Hirose, Takumi Maesaka and Taiki Watanabe}
\date{}

\begin{document}

\begin{abstract}
We prove two conjectures on the spaces generated by multiple $q$-zeta values.
More precisely, we show that the spaces $\mathcal{Z}_q^{\circ}$ and $\mathcal{Z}_{q,1}^{\circ}$ already generate the larger spaces $\mathcal{Z}_q$ and $\mathcal{Z}_{q,1}$, respectively.
Our result is stronger than the equality of $\mathbb{Q}$-vector spaces: for every generator in the larger spaces, we construct an explicit expression with integer coefficients in terms of the smaller generating families.
We first establish these formulas at the finite level, where suitable finite $q$-analogues admit recursive descriptions through generating series, and then pass to the infinite limit.
\end{abstract}

\maketitle

\section{Introduction}

The purpose of this paper is to prove two conjectures on the spaces of $q$-analogues of multiple zeta values ($q$-MZVs). More precisely, we establish the asserted equalities of $\mathbb{Q}$-vector spaces by giving explicit formulas with integer coefficients that express each relevant $q$-MZV in terms of the smaller families.

For $r\geq 1$, integers $k_1,\dots,k_r\ge 1$, polynomials $Q_1,\dots,Q_{r-1}\in\mathbb{Q}[X]$, and $Q_r\in X\mathbb{Q}[X]$ satisfying $\deg Q_j\le k_j$ for all $j$, define
\[
\zeta_q(k_1,\dots,k_r;Q_1,\dots,Q_r)
:=
\sum_{0<n_1<\cdots<n_r}
\prod_{j=1}^r \frac{Q_j(q^{n_j})}{(1-q^{n_j})^{k_j}}
\in \mathbb{Q}[[q]].
\]
For $r=0$, we set $\zeta_q(\emptyset;\emptyset):=1$. For $d\in\{0,1\}$, we define
\begin{align*}
\mathcal{Z}_{q,d}
&:=
\Span_{\mathbb{Q}}\bigl\{
\zeta_q(k_1,\dots,k_r;Q_1,\dots,Q_r)
\mid r\ge 0,\ k_j\ge 1,\ \deg Q_j\le k_j-d,\ Q_r\in X\mathbb{Q}[X]
\bigr\},\\
\mathcal{Z}_{q,d}^{\circ}
&:=
\Span_{\mathbb{Q}}\bigl\{
\zeta_q(k_1,\dots,k_r;Q_1,\dots,Q_r)
\mid r\ge 0,\ k_j\ge 1,\ \deg Q_j\le k_j-d,\ Q_j\in X\mathbb{Q}[X]
\bigr\}.
\end{align*}
We write $\mathcal{Z}_q:=\mathcal{Z}_{q,0}$ and $\mathcal{Z}_q^{\circ}:=\mathcal{Z}_{q,0}^{\circ}$.

The two conjectures proved in this paper assert the following equalities:
\begin{align*}
\mathcal{Z}_q^{\circ}&=\mathcal{Z}_q,\\
\mathcal{Z}_{q,1}^{\circ}&=\mathcal{Z}_{q,1}.
\end{align*}
The inclusion $\subset$ is immediate from the definitions in both cases. The first conjecture was proposed by Bachmann; the earliest explicit formulation seems to be \cite[Conjecture~4.3]{Bachmann_Bibrackets}, and the notations $\mathcal{Z}_q^{\circ}$ and $\mathcal{Z}_q$ are used in \cite[Conjecture~1]{Brindle_qmzv}. The second conjecture was also suggested by Bachmann and is recorded in \cite[Conjecture~4.1]{Brindle_qduality} and \cite[Figure~1]{Brindle_qmzv}.

To state our results in a form suited to explicit formulas, we introduce the models used throughout the paper.
We define the sets of indices as follows:
\begin{align*}
\mathbb{I} & :=\{(k_{1},\dots,k_{r})\in\mathbb{Z}_{\geq1}^{r}\,\mid\,r\geq0\},\\
\mathbb{I}^{\mathrm{adm}} & :=\{(k_{1},\dots,k_{r})\in\mathbb{I}\,\mid\,r=0\text{ or }k_{r}\neq1\},\\
\bar{\mathbb{I}} & :=\{(k_{1},\dots,k_{r})\in(\{\bar{1}\}\cup\mathbb{Z}_{\geq1})^{r}\,\mid\,r\geq0\},\\
\bar{\mathbb{I}}^{\mathrm{adm}} & :=\{(k_{1},\dots,k_{r})\in\bar{\mathbb{I}}\,\mid\,r=0\text{ or }k_{r}\neq\bar{1}\},
\end{align*}
where $\bar{1}$ is a formal symbol.
We define two models of $q$-MZVs by
\begin{align*}
\zeta_{q}^{\dagger}(\boldsymbol{k}) & :=\sum_{\substack{0<n_{1}\leq\cdots\leq n_{r}\\
n_{i}<n_{i+1}\ (k_{i}\neq\bar{1})
}
}\prod_{\substack{1\leq j\leq r\\
k_{j}\neq\bar{1}
}
}\frac{q^{n_{j}}}{(1-q^{n_{j}})^{k_{j}}}
\qquad(\boldsymbol{k}=(k_{1},\dots,k_{r})\in\bar{\mathbb{I}}^{\mathrm{adm}}),\\
\zeta_{q}^{\BZ}(\boldsymbol{k}) & :=\sum_{0<n_{1}<\cdots<n_{r}}\prod_{j=1}^{r}\frac{q^{n_{j}(k_{j}-1)}}{(1-q^{n_{j}})^{k_{j}}}\qquad(\boldsymbol{k}=(k_{1},\dots,k_{r})\in\mathbb{I}^{\mathrm{adm}}).
\end{align*}
Here, the notation $\zeta_{q}^{\dagger}(\boldsymbol{k})$ is introduced in this paper, and $\zeta_{q}^{\BZ}(\boldsymbol{k})$ is the Bradley--Zhao (BZ) model introduced independently by Bradley \cite{Bradley_qmzv} and Zhao \cite{Zhao_qmzv}.

These models are useful because they describe the spaces appearing in the conjectures; see \cref{lem:spaces_qMZV} below. Namely,
\begin{align*}
\mathcal{Z}_{q} & =\Span_{\mathbb{Q}}\{\zeta_{q}^{\dagger}(\boldsymbol{k})\,|\,\boldsymbol{k}\in\bar{\mathbb{I}}^{\mathrm{adm}}\},\\
\mathcal{Z}_{q}^{\circ} & =\Span_{\mathbb{Q}}\{\zeta_{q}^{\dagger}(\boldsymbol{k})\,|\,\boldsymbol{k}\in\bar{\mathbb{I}}^{\mathrm{adm}},k_{i}\neq\bar{1}\},\\
\mathcal{Z}_{q,1} & =\Span_{\mathbb{Q}}\{\zeta_{q}^{\BZ}(\boldsymbol{k})\,|\,\boldsymbol{k}\in\mathbb{I}^{\mathrm{adm}}\},\\
\mathcal{Z}_{q,1}^{\circ} & =\Span_{\mathbb{Q}}\{\zeta_{q}^{\BZ}(\boldsymbol{k})\,|\,\boldsymbol{k}\in\mathbb{I}^{\mathrm{adm}},k_{i}\neq1\}.
\end{align*}
Thus the conjectures say that every element in the first or third spanning family can be rewritten in terms of the second or fourth family, respectively. Our theorem does this with explicit integer coefficients.

To encode these expansions, we work in the noncommutative polynomial algebra
\[
\h:=\mathbb{Z}\langle x,y\rangle,
\qquad
\h^1:=\mathbb{Z}+y\h,
\qquad
\h^0:=\mathbb{Z}+y\h x,
\qquad
\h^{\ge 2}:=\mathbb{Z}+yx\,\mathbb{Z}\langle x,yx\rangle.
\]
The role of this algebra is to package families of $q$-MZVs into words on which our recursive formulas are stated. Define $\mathbb{Z}$-linear maps
\[
Z_q^{\dagger}:\h^1\to\mathbb{Q}[[q]],
\qquad
Z_q^{\BZ}:\h^0\to\mathbb{Q}[[q]]
\]
by $Z_q^{\dagger}(1):=Z_q^{\BZ}(1):=1$ and
\begin{align*}
Z_q^{\dagger}(yx^{k_1-1}\cdots yx^{k_r-1})&:=\zeta_q^{\dagger}(k_1,\dots,k_r),\\
Z_q^{\BZ}(yx^{k_1-1}\cdots yx^{k_r-1})&:=\zeta_q^{\BZ}(k_1,\dots,k_r).
\end{align*}
We construct explicit words $\mathfrak{E}_q(c_1,\dots,c_{2r})\in\h^1$ and $\mathfrak{D}_q(c_1,\dots,c_{2r})\in\h^{\ge 2}$ recursively and prove the following identities.

\begin{theorem}\label{thm:main_inf}
For $r\ge 0$ and $k_1,\dots,k_r, l_1,\dots,l_r\in\mathbb{Z}_{\ge 1}^r$, we have
\begin{align*}
\zeta_{q}^{\dagger}(\underbrace{\bar{1},\dots,\bar{1}}_{l_{1}-1},k_{1},\dots,\underbrace{\bar{1},\dots,\bar{1}}_{l_{r}-1},k_{r}) & =Z_{q}^{\dagger}\bigl(\mathfrak{E}_{q}(l_{1},k_{1},\dots,l_{r},k_{r})\bigr),\\
\zeta_{q}^{\BZ}(\underbrace{1,\dots,1}_{l_{1}-1},k_{1}+1,\dots,\underbrace{1,\dots,1}_{l_{r}-1},k_{r}+1) & =Z_{q}^{\BZ}\bigl(\mathfrak{D}_{q}(l_{1},k_{1},\dots,l_{r},k_{r})\bigr).
\end{align*}
\end{theorem}

This theorem already implies both conjectures. Indeed, the first identity expresses every generator of $\mathcal{Z}_{q}$ in terms of $\dagger$-values with no entry equal to $\bar{1}$, while the second expresses every admissible Bradley--Zhao value in terms of Bradley--Zhao values with all entries at least $2$.
The proof is obtained by first proving the finite analogues and then taking their limits.

For $N\ge 1$, define finite analogues by
\begin{align*}
\zeta_{q,N}^{\dagger}(\boldsymbol{k}) & :=\sum_{\substack{0<n_{1}\leq\cdots\leq n_{r}<N\\
n_{i}<n_{i+1}\ (k_{i}\neq\bar{1})
}}
\Bigg(\prod_{\substack{1\leq j\leq r\\
k_{j}=\bar{1}
}
}\frac{1}{1-q^{N-n_{j}}}
\Bigg)\Bigg(\prod_{\substack{1\leq j\leq r\\
k_{j}\neq\bar{1}
}
}\frac{q^{n_{j}}}{(1-q^{n_{j}})^{k_{j}}}\Bigg)
\qquad(\boldsymbol{k}=(k_{1},\dots,k_{r})\in\bar{\mathbb{I}}^{\mathrm{adm}}),\\
\zeta_{q,N}^{\BZ}(\boldsymbol{k}) & :=\sum_{0<n_{1}<\cdots<n_{r}<N}\prod_{j=1}^{r}\frac{q^{n_{j}(k_{j}-1)}}{(1-q^{n_{j}})^{k_{j}}}\qquad(\boldsymbol{k}=(k_{1},\dots,k_{r})\in\mathbb{I}).
\end{align*}
As $N\to\infty$, $\zeta_{q,N}^{\dagger}(\boldsymbol{k})$ converges to
$\zeta_q^{\dagger}(\boldsymbol{k})$, and when
$\boldsymbol{k}\in\mathbb{I}^{\mathrm{adm}}$,
$\zeta_{q,N}^{\BZ}(\boldsymbol{k})$ converges to
$\zeta_q^{\BZ}(\boldsymbol{k})$.
When $\boldsymbol{k}\in \mathbb{I}^\mathrm{adm}$, we also introduce another finite model
\[
\zeta_{q,N}^{\ds,\BZ}(\boldsymbol{k})
:=
\sum_{A\subset\{i\mid k_i=1\}}
\sum_{\substack{0<n_1\le \cdots\le n_r<N\\
n_i<n_{i+1}\quad(i\notin A)}}
\left(\prod_{i\in A}\frac{q^{N-n_i}}{1-q^{N-n_i}}\right)
\left(\prod_{i\notin A}\frac{q^{n_i(k_i-1)}}{(1-q^{n_i})^{k_i}}\right),
\]
which also converges to $\zeta_{q}^{\BZ}(\boldsymbol{k})$ as $N\to\infty$.
We further define $\mathbb{Z}$-linear maps
\[
Z_{q,N}^{\dagger}, Z_{q,N}^{\BZ}:\h^1\to\mathbb{Z}[[q]]
\]
by $Z_{q,N}^{\dagger}(1):=Z_{q,N}^{\BZ}(1):=1$ and
\begin{align*}
Z_{q,N}^{\dagger}(yx^{k_1-1}\cdots yx^{k_r-1})&:=\zeta_{q,N}^{\dagger}(k_1,\dots,k_r),\\
Z_{q,N}^{\BZ}(yx^{k_1-1}\cdots yx^{k_r-1})&:=\zeta_{q,N}^{\BZ}(k_1,\dots,k_r).
\end{align*}
Then the finite refinement of \cref{thm:main_inf} is as follows.

\begin{theorem}\label{thm:main_finite}
For $N\ge 1$,$r\ge 0$ and $\boldsymbol{k}=(k_1,\dots,k_r)$, $\boldsymbol{l}=(l_1,\dots,l_r)\in\mathbb{Z}_{\ge 1}^r$, we have
\begin{align*}
\zeta_{q,N}^{\dagger}(\underbrace{\bar{1},\dots,\bar{1}}_{l_{1}-1},k_{1},\dots,\underbrace{\bar{1},\dots,\bar{1}}_{l_{r}-1},k_{r}) & =Z_{q,N}^{\dagger}\bigl(\mathfrak{E}_{q}(l_{1},k_{1},\dots,l_{r},k_{r})\bigr),\\
\zeta_{q,N}^{\ds,\BZ}(\underbrace{1,\dots,1}_{l_{1}-1},k_{1}+1,\dots,\underbrace{1,\dots,1}_{l_{r}-1},k_{r}+1) & =Z_{q,N}^{\BZ}\bigl(\mathfrak{D}_{q}(l_{1},k_{1},\dots,l_{r},k_{r})\bigr).
\end{align*}
\end{theorem}

Passing to the limit as $N\to\infty$ in this finite-level result recovers the corresponding infinite-level result.
Furthermore, writing \(\mathbb{N}:=\{1,2,\ldots\}\), the maps
\[
Z_{q,\mathbb{N}}^{\dagger}(u):=(Z_{q,N}^{\dagger}(u))_{N\in\mathbb{N}},
\qquad
Z_{q,\mathbb{N}}^{\BZ}(u):=(Z_{q,N}^{\BZ}(u))_{N\in\mathbb{N}}
\]
are injective; see \cref{lem:independence}.
Thus, unlike their infinite-level counterparts, the finite-level identities
also characterize the recursively defined words
$\mathfrak{E}_q$ and $\mathfrak{D}_q$.

The proof strategy for \cref{thm:main_finite}, which also yields
\cref{thm:main_inf}, $\mathcal{Z}_q^{\circ}=\mathcal{Z}_q$, and
$\mathcal{Z}_{q,1}^{\circ}=\mathcal{Z}_{q,1}$, is as follows.
We treat the two finite identities in \cref{thm:main_finite} in a unified
framework.  The relevant finite sums are packaged into generating functions,
for which we establish recurrence relations.  Comparing coefficients then
gives recurrences for the finite sums themselves.  The recursive definitions
of $\mathfrak{E}_q$ and $\mathfrak{D}_q$ are chosen so that the corresponding
finite $q$-series satisfy the same recurrences and initial conditions.  Hence
the finite identities follow by induction.

\begin{remark}
The finite sums $\zeta_{q,N}^\dagger(\boldsymbol{k})$ and $\zeta_{q,N}^{\ds,\BZ}\left(\boldsymbol{k}\right)$ are $q$-analogues of the quantities introduced in \cite{HMSW}:
\[
\zeta_{N}\binom{\boldsymbol{l}}{\boldsymbol{k}}
:=
\sum_{\substack{0<n_{1,1}\le\cdots\le n_{1,l_{1}}\\
<n_{2,1}\le\cdots\le n_{2,l_{2}}\\
\cdots\\
<n_{r,1}\le\cdots\le n_{r,l_{r}}<N}}
\prod_{j=1}^{r}
\frac{1}{(N-n_{j,1})\cdots (N-n_{j,l_{j}-1})n_{j,l_{j}}^{k_{j}}},
\]
\[
\zeta_{N}^{\ds}(\boldsymbol{k})
:=\sum_{A\subset\{i\mid k_{i}=1\}}
\sum_{\substack{0<n_{1}\le\cdots\le n_{r}<N\\
n_{i}<n_{i+1}\quad(i\notin A)}}
\left(\prod_{i\in A}\frac{1}{N-n_{i}}\right)
\left(\prod_{i\notin A}\frac{1}{n_{i}^{k_{i}}}\right).
\]
More precisely, with the weights defined by
\[
\wt(\boldsymbol{k},\boldsymbol{l}):=\sum_{j=1}^{r}(k_{j}+l_{j}-1),
\qquad
\wt(\boldsymbol{k}):=\sum_{j=1}^{r}k_{j},
\]
we have
\[
\lim_{q\to1}(1-q)^{\wt(\boldsymbol{k},\boldsymbol{l})}
\zeta_{q,N}^{\dagger}
(\underbrace{\bar{1},\dots,\bar{1}}_{l_{1}-1},k_{1},\dots,
\underbrace{\bar{1},\dots,\bar{1}}_{l_{r}-1},k_{r})
=
\zeta_{N}\binom{\boldsymbol{l}}{\boldsymbol{k}},
\]
and
\[
\lim_{q\to 1}(1-q)^{\wt(\boldsymbol{k})}
\zeta_{q,N}^{\ds,\BZ}(\boldsymbol{k})
=
\zeta_{N}^{\ds}(\boldsymbol{k}).
\]
Moreover, if $k_{1},\dots,k_{r}\ge 2$, then
\[
\zeta_{q,N}^{\ds,\BZ}(\boldsymbol{k})=\zeta_{q,N}^{\BZ}(\boldsymbol{k}).
\]

\end{remark}

\begin{remark}
By symmetry in the definitions of $\mathfrak{D}_{q}$ and $\mathfrak{E}_{q}$, one has
\begin{align*}
\mathfrak{D}_{q}(l_{1},k_{1},\dots,l_{r},k_{r}) & =\mathfrak{D}_{q}(k_{r},l_{r},\dots,k_{1},l_{1}),\\
\mathfrak{E}_{q}(l_{1},k_{1},\dots,l_{r},k_{r}) & =\mathfrak{E}_{q}(k_{r},l_{r},\dots,k_{1},l_{1}).
\end{align*}
Hence the theorem includes the following identities:
\begin{align}
\zeta_{q,N}^{\dagger}(\{\bar{1}\}^{l_{1}-1},k_{1},\dots,\{\bar{1}\}^{l_{r}-1},k_{r}) & =\zeta_{q,N}^{\dagger}(\{\bar{1}\}^{k_{r}-1},l_{r},\dots,\{\bar{1}\}^{k_{1}-1},l_{1}),\label{eq:dual_zeta_qN_flat}\\
\zeta_{q,N}^{\ds,\BZ}(\{1\}^{l_{1}-1},k_{1}+1,\dots,\{1\}^{l_{r}-1},k_{r}+1) & =\zeta_{q,N}^{\ds,\BZ}(\{1\}^{k_{r}-1},l_{r}+1,\dots,\{1\}^{k_{1}-1},l_{1}+1).\label{eq:dual_zeta_qN_diamond}
\end{align}
Moreover, setting $l_{1}=\cdots=l_{r}=1$ in (\ref{eq:dual_zeta_qN_flat}) and then applying the change of variables $n\mapsto N-n$ on the right-hand side, we obtain
\[
\sum_{0<n_{1}<\cdots<n_{r}<N}\prod_{j=1}^{r}\frac{q^{n_{j}}}{(1-q^{n_{j}})^{k_{j}}}
=
\sum_{\substack{0<n_{1,1}\leq\cdots\leq n_{1,k_{1}}\\
<n_{2,1}\leq\cdots\leq n_{2,k_{2}}\\
\cdots\\
<n_{r,1}\leq\cdots\leq n_{r,k_{r}}<N
}
}\prod_{j=1}^{r}\frac{q^{N-n_{j,1}}}{1-q^{N-n_{j,1}}}\frac{1}{1-q^{n_{j,2}}}\cdots\frac{1}{1-q^{n_{j,k_{j}}}}.
\]
This is a $q$-analogue of the MSW formula \cite[Theorem 1.3]{MSW}. After the substitution $q\mapsto q^{-1}$, it coincides with Tsuruta's $q$-analogue of the MSW formula \cite[Theorem 1.2]{Tsuruta_qMSW}. Furthermore, (\ref{eq:dual_zeta_qN_diamond}) is a $q$-analogue of the duality for $\zeta_{N}^{\ds}$ (\cite{HMSW}), and in the limit $N\to\infty$ it yields the duality for  $\zeta_q^{\BZ}$.
\end{remark}

\begin{remark}
The equality $\mathcal{Z}_{q}=\mathcal{Z}_{q}^{\circ}$ also admits an interpretation in terms of polynomial functions on partitions. 
See \cite{Brindle_qmzv, Partition}.
\end{remark}

\subsection*{Acknowledgments}
The authors are grateful to Henrik Bachmann for helpful comments on this research.
This research was supported by JSPS KAKENHI Grant Numbers JP22K03244.

\section{Preliminaries}\label{sec:preliminaries}

Before turning to the proofs, we collect some combinatorial notation and a few basic lemmas.

\subsection{Notation for subsets of $\{1,\dots,n\}$}

For $n\ge 0$, let $[n]$ denote the set $\{1,\dots,n\}$. For $r\in\mathbb{Z}_{\ge 0}$, define
\[
\mathrm{EO}(r):=\left\{S\subset [2r]\,\middle|\, \exists R\subset [r-1]\text{ such that }S=\bigcup_{j\in R}\{2j,2j+1\}\right\},
\]
\[
\mathrm{OE}(r):=\left\{S\subset [2r]\,\middle|\, \exists R\subset [r]\text{ such that }S=\bigcup_{j\in R}\{2j-1,2j\}\right\},
\]
\[
\mathrm{OE}^*(r):=\{S\in\mathrm{OE}(r)\mid \forall i,j\in S,\ |i-j|\neq 2\},
\]
\[
\mathcal{T}(r):=\{S\cup S'\subset [2r]\mid S\in\mathrm{EO}(r),\ S'\in\mathrm{OE}^*(r),\ S\cap S'=\emptyset\}.
\]
For $T\in\mathcal{T}(r)$, we put
$\kappa(T):=\frac{\#T}{2}.$
Equivalently, $\mathcal{T}(r)$ is the set of partial domino tilings of a row of length $2r$, numbered from left to right by $1,\dots,2r$, in which no two odd-even dominoes are adjacent, and $\kappa(T)$ denotes the number of dominoes used in the partial tiling $T$.

For a finite subset $S$ of $\mathbb{Z}_{>0}$, define
\[
\mathrm{eo}(S):=\#\{ j\in\mathbb{Z}_{>0} \mid \{2j,2j+1\}\subset S \},
\]
\[
\mathrm{oe}(S):=\#\{ j\in\mathbb{Z}_{>0} \mid \{2j-1,2j\}\subset S \},
\]
\[
\sigma_{S}:\mathbb{Z}_{>0}\setminus S\xrightarrow{\sim}\mathbb{Z}_{>0}\quad;\quad i\mapsto\#([i]\setminus S).
\]

For $\boldsymbol{c}=(c_1,\dots,c_n)\in\mathbb{Z}_{\ge 1}^n$, let
\[
[n]_{\boldsymbol{c}}^1:=\{i\in[n]\mid c_i=1\},
\qquad
[n]_{\boldsymbol{c}}^{>1}:=\{i\in[n]\mid c_i>1\}=[n]\setminus [n]_{\boldsymbol{c}}^1.
\]
If $A\subset [n]_{\boldsymbol{c}}^1$ and $B\subset [n]_{\boldsymbol{c}}^{>1}$, we write $\boldsymbol{c}_{A,B}\in\mathbb{Z}_{\ge 1}^{n-\#A}$ for the sequence obtained from $\boldsymbol{c}$
by subtracting $1$ from the entries indexed by $B$ and deleting the entries indexed by $A$.

For $B\subset [2s]$, define
\begin{align*}
\alpha(B)
&:= \#\{i\in\{0,\dots,s\}\mid B\cap\{2i,2i+1\}\neq\emptyset\} \\
&= \#B-\mathrm{eo}(B),\\
\beta(B)
&:= \#\{i\in[s]\mid B\cap\{2i-1,2i\}\neq\emptyset\} \\
&= \#B-\mathrm{oe}(B).
\end{align*}
Thus $\alpha(B)$ (respectively, $\beta(B)$) is the minimal number of even-odd (respectively, odd-even) dominoes needed to cover $B$.

If $A\in\mathcal{T}(s)$ and $B\subset[2s]$ with $A\cap B=\emptyset$, define
\[
\beta_A(B):=\beta(\sigma_A(B)).
\]
In other words, $\beta_A(B)$ is the minimal number of odd-even dominoes needed to cover $B$ after the positions indexed by $A$ have been removed.

Finally, for a sequence $(W_1,\dots,W_n)$ and a subset $T\subset\{1,\dots,n\}$, we write
\[
(W_1,\dots,W_n)_{\setminus T}
\]
for the sequence obtained by deleting the entries with indices in $T$. This notation will be used in the generating-series identities.

\subsection{Auxiliary models and spanning lemmas}

We first recall the Schlesinger--Zudilin (SZ) model, which serves as a convenient bridge between the $\dagger$-model and the spaces $\mathcal{Z}_q$ and $\mathcal{Z}_q^\circ$.
\begin{definition}
For $\boldsymbol{k}=(k_1,\dots,k_r)\in\mathbb{Z}_{\ge 0}^r$ with either $r=0$ or $k_r>0$, define the SZ-model of $q$-MZVs by $\zeta_q^{\SZ}(\emptyset):=1$ and
\[
\zeta_q^{\SZ}(k_1,\dots,k_r)
:=
\sum_{0<n_1<\cdots<n_r}
\prod_{j=1}^r \frac{q^{n_jk_j}}{(1-q^{n_j})^{k_j}}
\in \mathbb{Q}[[q]]
\qquad (r\ge 1).
\]
\end{definition}

The following lemma expresses the spaces $\mathcal{Z}_q$ and $\mathcal{Z}_q^\circ$ in terms of the SZ $q$-MZVs.

\begin{lemma}[cf.\ \cite{Bachmann_Kuehn_dimension_conjecture,Brindle_qmzv}]\label{lem:space_SZ}
We have
\[
\mathcal{Z}_q
=
\Span_{\mathbb{Q}}\bigl\{\zeta_q^{\SZ}(k_1,\dots,k_r)
\mid r\ge 0,\ k_r\ge 1,\ k_i\ge 0\bigr\},
\]
\[
\mathcal{Z}_q^{\circ}
=
\Span_{\mathbb{Q}}\bigl\{\zeta_q^{\SZ}(k_1,\dots,k_r)
\mid r\ge 0,\ k_i\ge 1\bigr\}.
\]
\end{lemma}

Next, we give transformation formulas between the SZ-model and the $\dagger$-model.
\begin{lemma}\label{lem:transformation_between_models}
For $r\ge 0$ and $\boldsymbol{l}=(l_1,\dots,l_r),\boldsymbol{k}=(k_1,\dots,k_r)\in\mathbb{Z}_{\ge 1}^r$, we have
\begin{align*}
\zeta_{q}^{\SZ}(\{0\}^{l_{1}-1},k_{1},\dots,\{0\}^{l_{r}-1},k_{r}) & =\sum_{\boldsymbol{l}',\boldsymbol{k}'\in\mathbb{Z}_{\ge1}^{r}}\bar{b}(\boldsymbol{l};\boldsymbol{l}')\bar{b}(\boldsymbol{k};\boldsymbol{k}')\zeta_{q}^{\dagger}(\{\bar{1}\}^{l_{1}'-1},k_{1}',\dots,\{\bar{1}\}^{l_{r}'-1},k_{r}'),\\
\zeta_{q}^{\SZ}(\boldsymbol{k}) & =\sum_{\boldsymbol{k}'\in\mathbb{Z}_{\ge1}^{r}}\bar{b}(\boldsymbol{k};\boldsymbol{k}')\zeta_{q}^{\dagger}(\boldsymbol{k}'),\\
\zeta_{q}^{\dagger}(\{\bar{1}\}^{l_{1}-1},k_{1},\dots,\{\bar{1}\}^{l_{r}-1},k_{r}) & =\sum_{\boldsymbol{l}',\boldsymbol{k}'\in\mathbb{Z}_{\ge1}^{r}}b(\boldsymbol{l};\boldsymbol{l}')b(\boldsymbol{k};\boldsymbol{k}')\zeta_{q}^{\SZ}(\{0\}^{l_{1}'-1},k_{1}',\dots,\{0\}^{l_{r}'-1},k_{r}'),\\
\zeta_{q}^{\dagger}(\boldsymbol{k}) & =\sum_{\boldsymbol{k}'\in\mathbb{Z}_{\ge1}^{r}}b(\boldsymbol{k};\boldsymbol{k}')\zeta_{q}^{\SZ}(\boldsymbol{k}').
\end{align*}
Here
\[
b(\boldsymbol{m};\boldsymbol{m}'):=\prod_{j=1}^{r}\binom{m_{j}-1}{m_{j}'-1},\quad\bar{b}(\boldsymbol{m};\boldsymbol{m}'):=\prod_{j=1}^{r}(-1)^{m_{j}-m_{j}'}\binom{m_{j}-1}{m_{j}'-1}\qquad(\boldsymbol{m},\boldsymbol{m}'\in\mathbb{Z}_{\ge1}^{r}).
\]
Each sum is effectively finite, since $b(\boldsymbol{m};\boldsymbol{m}')=\bar{b}(\boldsymbol{m};\boldsymbol{m}')=0$ unless $m_{j}'\le m_{j}$ for all $j$.
\end{lemma}

\begin{proof}
By definition,
\begin{align*}
\zeta_{q}^{\SZ}(\{0\}^{l_{1}-1},k_{1},\dots,\{0\}^{l_{r}-1},k_{r}) & =\sum_{0=n_{0}<n_{1}<\cdots<n_{r}}\prod_{j=1}^{r}\binom{n_{j}-n_{j-1}-1}{l_{j}-1}\frac{q^{n_{j}k_{j}}}{(1-q^{n_{j}})^{k_{j}}},\\
\zeta_{q}^{\dagger}(\{\bar{1}\}^{l_{1}-1},k_{1},\dots,\{\bar{1}\}^{l_{r}-1},k_{r}) & =\sum_{0=n_{0}<n_{1}<\cdots<n_{r}}\prod_{j=1}^{r}\binom{n_{j}-n_{j-1}+l_{j}-2}{l_{j}-1}\frac{q^{n_{j}}}{(1-q^{n_{j}})^{k_{j}}}.
\end{align*}
Using the identities
\[
\binom{n_{j}-n_{j-1}-1}{l_{j}-1}
=
\sum_{l_{j}'=1}^{l_{j}}(-1)^{l_{j}-l_{j}'}
\binom{l_{j}-1}{l_{j}'-1}
\binom{n_{j}-n_{j-1}+l_{j}'-2}{l_{j}'-1}
\]
and
\[
\frac{q^{n_{j}k_{j}}}{(1-q^{n_{j}})^{k_{j}}}
=
\sum_{k_{j}'=1}^{k_{j}}(-1)^{k_{j}-k_{j}'}
\binom{k_{j}-1}{k_{j}'-1}
\frac{q^{n_{j}}}{(1-q^{n_{j}})^{k_{j}'}},
\]
we obtain the first formula. The second formula is the special case $l_{1}=\cdots=l_{r}=1$.

Similarly,
\[
\binom{n_{j}-n_{j-1}+l_{j}-2}{l_{j}-1}
=
\sum_{l_{j}'=1}^{l_{j}}\binom{l_{j}-1}{l_{j}'-1}\binom{n_{j}-n_{j-1}-1}{l_{j}'-1}
\]
and
\[
\frac{q^{n_{j}}}{(1-q^{n_{j}})^{k_{j}}}
=
\sum_{k_{j}'=1}^{k_{j}}\binom{k_{j}-1}{k_{j}'-1}\frac{q^{n_{j}k_{j}'}}{(1-q^{n_{j}})^{k_{j}'}}
\]
yield the third formula, and the fourth again follows by setting $l_{1}=\cdots=l_{r}=1$.
\end{proof}

The preceding lemma immediately yields the spanning statements used in the introduction.

\begin{lemma}\label{lem:spaces_qMZV}
We have
\begin{align*}
\mathcal{Z}_{q} & =\Span_{\mathbb{Q}}\{\zeta_{q}^{\dagger}(\boldsymbol{k})\,|\,\boldsymbol{k}\in\bar{\mathbb{I}}^{\mathrm{adm}}\},\\
\mathcal{Z}_{q}^{\circ} & =\Span_{\mathbb{Q}}\{\zeta_{q}^{\dagger}(\boldsymbol{k})\,|\,\boldsymbol{k}\in\bar{\mathbb{I}}^{\mathrm{adm}},k_{i}\neq\bar{1}\},\\
\mathcal{Z}_{q,1} & =\Span_{\mathbb{Q}}\{\zeta_{q}^{\BZ}(\boldsymbol{k})\,|\,\boldsymbol{k}\in\mathbb{I}^{\mathrm{adm}}\},\\
\mathcal{Z}_{q,1}^{\circ} & =\Span_{\mathbb{Q}}\{\zeta_{q}^{\BZ}(\boldsymbol{k})\,|\,\boldsymbol{k}\in\mathbb{I}^{\mathrm{adm}},k_{i}\neq1\}.
\end{align*}
\end{lemma}

\begin{proof}
The first two identities follow from \cref{lem:space_SZ,lem:transformation_between_models}. The last two are exactly \cite[Proposition~4]{Brindle_qmzv} and \cite[Proposition~16]{Brindle_qmzv}.
\end{proof}

\begin{remark}
It is also known that $\mathcal{Z}_{q,d}$ and $\mathcal{Z}_{q,d}^{\circ}$ can be described in terms of another model, denoted by $g$.
Moreover, \cref{lem:spaces_qMZV} can also be proved using the transformation formula between the generating functions of $g$ and $\zeta_{q}^{\dagger}$ in \cref{rem:genfunc_Gg}.
\end{remark}

\subsection{Linear independence of finite-level $q$-MZVs}

By the following lemma, the first (resp.\ second) identity in \cref{thm:main_finite} characterizes $\mathfrak{E}_{q}$ (resp.\ $\mathfrak{D}_{q}$).

\begin{lemma}\label{lem:independence}
Define $\mathbb{Z}$-linear maps $Z_{q,\mathbb{N}}^{\dagger},Z_{q,\mathbb{N}}^{\BZ}:\h^1\to\mathbb{Z}[[q]]^{\mathbb{N}}$ by
\[
Z_{q,\mathbb{N}}^{\dagger}(u):=(Z_{q,N}^{\dagger}(u))_{N\in\mathbb{N}},
\qquad
Z_{q,\mathbb{N}}^{\BZ}(u):=(Z_{q,N}^{\BZ}(u))_{N\in\mathbb{N}}.
\]
Then both maps are injective.
\end{lemma}

\begin{proof}
Let $Z_N:\h^1\to\mathbb{Q}$ and $Z_{\mathbb{N}}:\h^1\to\mathbb{Q}^{\mathbb{N}}$ be defined by
\begin{align*}
Z_N(yx^{k_1-1}\cdots yx^{k_r-1})
&:=
\sum_{0<n_1<\cdots<n_r<N}
\frac{1}{n_1^{k_1}\cdots n_r^{k_r}},\\
Z_{\mathbb{N}}(u)&:=(Z_N(u))_{N\in\mathbb{N}}.
\end{align*}
By the linear independence of finite harmonic sums, $Z_{\mathbb{N}}$ is injective; see \cite[Theorem~3.1]{Yamamoto_injectivity}. If $k_1+\cdots+k_r\le K$, then
\begin{align*}
\lim_{q\to 1}(1-q)^K\zeta_{q,N}^{\dagger}(k_1,\dots,k_r)
&=
\lim_{q\to 1}(1-q)^K\zeta_{q,N}^{\BZ}(k_1,\dots,k_r)\\
&=
\begin{cases}
\displaystyle \sum_{0<n_1<\cdots<n_r<N}\frac{1}{n_1^{k_1}\cdots n_r^{k_r}} & \text{if }k_1+\cdots+k_r=K,\\[1ex]
0 & \text{if }k_1+\cdots+k_r<K.
\end{cases}
\end{align*}
Now let $u\in\h^1\setminus\{0\}$ and write $u=\sum_{k=0}^K u_k$, where $u_k$ denotes the homogeneous part of degree $k$ and $u_K\neq 0$. If $Z_{q,\mathbb{N}}^{\dagger}(u)=0$, then for every $N$,
\[
0=\lim_{q\to 1}(1-q)^K Z_{q,N}^{\dagger}(u)=Z_N(u_K).
\]
Since $Z_{\mathbb{N}}$ is injective, we get $u_K=0$, a contradiction. Hence $Z_{q,\mathbb{N}}^{\dagger}$ is injective. The proof for $Z_{q,\mathbb{N}}^{\BZ}$ is identical.
\end{proof}

\section{Proof of the main theorem}

We now turn to the proof of the main theorem. The strategy is to encode the truncated $q$-MZVs in a family of generating series,
derive recurrence relations for these series, and then read off a recursive formula for the words that realize the corresponding $q$-series.
In this section, we set $\qq{n}:=1-q^n$.
To present the proof in parallel, we use the following finite model
\[
\zeta_{q,N}^{\ds,\dagger}(\boldsymbol{k})
:=
(-1)^{\wt(\bm k)}\zeta_{q^{-1},N}^{\ds,\BZ}(\boldsymbol{k})=
\sum_{A\subset\{i\mid k_i=1\}}
\sum_{\substack{0<n_1\le \cdots\le n_r<N\\
n_i<n_{i+1}\quad(i\notin A)}}
\left(\prod_{i\in A}\frac{1}{\qq{N-n_i}}\right)
\left(\prod_{i\notin A}\frac{q^{n_i}}{\qq{n_i}^{k_i}}\right)
\]
in place of $\zeta_{q,N}^{\ds,\BZ}$.

\subsection{Generating functions}
In this subsection, we introduce generating functions for the truncated $q$-MZVs and prove recurrence relations for them.
First, we introduce an auxiliary left-truncation parameter $M\in \{0,\dots,N-1\}$ by setting
\begin{align*}
\zeta_{q,M,N}^{\dagger}(\boldsymbol{k}) & :=\sum_{\substack{M<n_{1}\leq\cdots\leq n_{r}<N\\
n_{i}<n_{i+1}\ (k_{i}\neq\bar{1})
}
}\Bigg(\prod_{\substack{1\leq j\leq r\\
k_{j}=\bar{1}
}
}\frac{1}{\qq{N-n_{j}}}\Bigg)\Bigg(\prod_{\substack{1\leq j\leq r\\
k_{j}\neq\bar{1}
}
}\frac{q^{n_{j}}}{\qq{n_{j}}^{k_{j}}}\Bigg),\\
\zeta_{q,M,N}^{\ds,\dagger}(\boldsymbol{k}) & :=\sum_{A\subset\{i\mid k_{i}=1\}}\sum_{\substack{M<n_{1}\le\cdots\le n_{r}<N\\
n_{i}<n_{i+1}\quad(i\notin A)
}
}\left(\prod_{i\in A}\frac{1}{\qq{N-n_{i}}}\right)\left(\prod_{i\notin A}\frac{q^{n_{i}}}{\qq{n_{i}}^{k_{i}}}\right),
\end{align*}
where, in each case, $\boldsymbol{k}$ ranges over the same set as in
the corresponding definition without $M$.
Furthermore, for $l_{1},\dots,l_{r},k_{1},\dots,k_{r}\geq1$, $\epsilon\in\{0,1\}$, and $0\le M<N$, we put
\[
\xi_{q,M,N}^{\epsilon}(l_1,k_1,\dots,l_r,k_r):=\begin{cases}
\zeta_{q,M,N}^{\dagger}(\{\bar{1}\}^{l_{1}-1},k_{1},\dots,\{\bar{1}\}^{l_{r}-1},k_{r}) & \text{if }\epsilon=0\\
\zeta_{q,M,N}^{\ds,\dagger}(\{1\}^{l_{1}-1},k_{1}+1,\dots,\{1\}^{l_{r}-1},k_{r}+1) & \text{if }\epsilon=1.
\end{cases}
\]
Then, we define the generating functions of $\xi_{q,M,N}^{\epsilon}$ as follows.
\begin{definition}
For $N> M \ge 0$, $r\ge 0$ and $\epsilon\in \{0,1\}$, we define the following elements of
\[
\mathbb{Q}[[q,\qq{X_i},\qq{Y_i}\mid i=1,\dots,r]]
\]
by
\begin{align*}
G_{q,M,N}^{\epsilon}(Y_1,X_1,\dots,Y_r,X_r)
&:=
\sum_{\substack{k_1,\dots,k_r\ge 0\\ l_1,\dots,l_r\ge 0}}
\left(\prod_{j=1}^r \qq{Y_j}^{l_j}\qq{X_j}^{k_j}\right)
\xi_{q,M,N}^{\epsilon}(l_1+1,k_1+1,\dots,l_r+1,k_r+1),\\
\widetilde{G}_{q,M,N}^{\epsilon}(Y_1,X_1,\dots,Y_r,X_r)
&:=
\left(\prod_{i=1}^r(\qq{N}-\qq{X_i+Y_i})\right)
G_{q,M,N}^{\epsilon}(Y_1,X_1,\dots,Y_r,X_r).
\end{align*}
\end{definition}

Furthermore, we put $\xi_{q,N}^{\epsilon}(\cdots):=\xi_{q,0,N}^{\epsilon}(\cdots)$, $G_{q,N}^{\epsilon}(\cdots):=G_{q,0,N}^{\epsilon}(\cdots)$, and $\widetilde{G}_{q,N}^{\epsilon}(\cdots):=\widetilde{G}_{q,0,N}^{\epsilon}(\cdots)$.

\begin{remark}\label{rem:genfunc_Gg}
Following \cite{Bachmann_Bibrackets,CombMES}, we recall another model of $q$-MZVs,
\[
g\binom{k_{1},\dots,k_{r}}{d_{1},\dots,d_{r}}:=\sum_{\substack{0<m_{1}<\cdots<m_{r}\\
0<n_{1},\dots,n_{r}
}
}\prod_{j=1}^{r}\frac{n_{j}^{k_{j}-1}m_{j}^{d_{j}}q^{m_{j}n_{j}}}{(k_{j}-1)!}
\]
and its generating function,
\[
\mathfrak{g}\binom{X_{1},\dots,X_{r}}{Y_{1},\dots,Y_{r}}:=\sum_{\substack{k_{1},\dots,k_{r}>0\\
d_{1},\dots,d_{r}\geq 0
}
}g\binom{k_{1},\dots,k_{r}}{d_{1},\dots,d_{r}}\prod_{j=1}^{r}\frac{X_{j}^{k_{j}-1}Y_{j}^{d_{j}}}{d_{j}!}.
\]
Let $G_{q}^{\epsilon}(\cdots):=\lim_{N\to\infty}G_{q,N}^{\epsilon}(\cdots)$. Then $G_{q}^{0}$ and $\mathfrak{g}$ are related by the following identity:
\[
G_{q}^{0}(Y_{1}+\cdots+Y_{r},X_{1},\dots,Y_{r-1}+Y_{r},X_{r-1},Y_{r},X_{r})=\mathfrak{g}\binom{-X_{1}\log q,\dots,-X_{r}\log q}{-Y_{1}\log q,\dots,-Y_{r}\log q}.
\]
\end{remark}

Before stating the main theorem of this subsection, we prepare a lemma.
\begin{lemma}
For $\epsilon\in\{0,1\}$, $N>M>0$, $r\geq1$ and $\bX=(Y_{1},X_{1},\dots,Y_{r},X_{r})$,
we have
\begin{equation}
\begin{split}\frac{\qq{N-M}-\qq{Y_{1}}}{\qq{N-M}}G_{q,M-1,N}^{\epsilon}(\bX) & =\left(1+\frac{q^{M}\qq{Y_1}}{\qq{M}}\right)^{\epsilon}G_{q,M,N}^{\epsilon}(\bX)\\
 & \quad+\frac{q^{M}}{\qq{M}^{\epsilon}}\frac{1}{\qq{M}-\qq{X_{1}}}G_{q,M,N}^{\epsilon}(\bX_{\setminus\{1,2\}}).
\end{split}
\label{eq:G_diff}
\end{equation}
\end{lemma}

\begin{proof}
By definition, we have
\begin{align*}
 & \xi_{q,M-1,N}^{\epsilon}(l_{1},k_{1},\dots,l_{r},k_{r})-\xi_{q,M,N}^{\epsilon}(l_{1},k_{1},\dots,l_{r},k_{r})\\
 & =\begin{cases}
\frac{1}{\qq{N-M}}\xi_{q,M-1,N}^{\epsilon}(l_{1}-1,k_{1},\dots,l_{r},k_{r})+\frac{\epsilon q^{M}}{\qq{M}}\xi_{q,M,N}^{\epsilon}(l_{1}-1,k_{1},\dots,l_{r},k_{r}) & l_{1}>1\\
\frac{q^{M}}{\qq{M}^{k_{1}+\epsilon}}\xi_{q,M,N}^{\epsilon}(l_{2},k_{2},\dots,l_{r},k_{r}) & l_{1}=1,
\end{cases}
\end{align*}
and thus,
\begin{align*}
G_{q,M-1,N}^{\epsilon}(\bX)-G_{q,M,N}^{\epsilon}(\bX) & =\frac{\qq{Y_{1}}}{\qq{N-M}}G_{q,M-1,N}^{\epsilon}(\bX)+\frac{\epsilon q^{M}\qq{Y_{1}}}{\qq{M}}G_{q,M,N}^{\epsilon}(\bX)\\
 & \quad+\frac{q^{M}}{\qq{M}^{\epsilon}}\frac{1}{\qq{M}-\qq{X_{1}}}G_{q,M,N}^{\epsilon}(\bX_{\setminus\{1,2\}}).
\end{align*}
This implies the lemma.
\end{proof}

The following is the main theorem of this subsection.
\begin{theorem}\label{thm:recur_for_gen_withM}
For $\epsilon\in\{0,1\}$, $N>M\geq0$, $r\geq 0$ and $\bX=(Y_{1},X_{1},\dots,Y_{r},X_{r})$,
we have
\begin{align*}
 & \frac{\prod_{i=0}^{r}(\qq{N}-\qq{X_{i}+Y_{i+1}})}{\qq{N}-\qq{M}}G_{q,M,N+1}^{\epsilon}(\bX)\\
 & =\sum_{T\in\mathcal{T}(r)}(-1)^{\mathrm{eo}(T)}\left(\frac{q^{N}}{\qq{N}^{\epsilon}}\right)^{\kappa(T)}\widetilde{G}_{q,M,N}^{\epsilon}(\bX_{\setminus T}),
\end{align*}
where we understand that $X_{0}:=M$ and $Y_{r+1}:=0$.
\end{theorem}

\begin{corollary}
\label{cor:recur_for_gen}
For $\epsilon\in\{0,1\}$, $N>0$, $r\geq0$ and $\bX=(Y_{1},X_{1},\dots,Y_{r},X_{r})$,
we have
\begin{align*}
 & \frac{\prod^{r}_{i=0}(\qq{N}-\qq{X_{i}+Y_{i+1}})}{\qq{N}}G^{\epsilon}_{q,N+1}(\bX)\\
 & =\sum_{T\in\mathcal{T}(r)}(-1)^{\mathrm{eo}(T)}\left(\frac{q^{N}}{\qq{N}^{\epsilon}}\right)^{\kappa(T)}\widetilde{G}^{\epsilon}_{q,N}(\bX_{\setminus T}),
\end{align*}
where we understand that $X_{0}:=Y_{r+1}:=0$.
\end{corollary}

\begin{proof}[Proof of \Cref{thm:recur_for_gen_withM} and \Cref{cor:recur_for_gen}.]
We fix $\epsilon\in\{0,1\}$ and $N>0$. The case $r=0$ is obvious
since both sides are equal to $1$. We assume $r\geq1$. Let $L_{M}(\bX)$
(resp.\ $R_{M}(\bX)$) be the left-hand (resp.\ right-hand) side of
\Cref{thm:recur_for_gen_withM}. Put
\begin{align*}
A_{M}(Y) & :=\frac{\qq{N-M}}{\qq{N-M}-\qq Y}\left(1+\frac{q^{M}\qq Y}{\qq{M}}\right)^{\epsilon}\\
B_{M}(X,Y) & :=\frac{q^{M}}{\qq{M}^{\epsilon}}\frac{\qq{N-M}}{\qq{N-M}-\qq Y}\frac{\qq{N}-\qq{X+Y}}{\qq{M}-\qq X}.
\end{align*}
Note that they satisfy
\begin{equation}
B_{M}(X,Y)-B_{M}(X,Y')=\frac{q^{N}}{\qq{N}^{\epsilon}}\left(A_{M}(Y)-A_{M}(Y')\right).\label{eq:B_diff}
\end{equation}

By replacing $N$ by $N+1$ in  \eqref{eq:G_diff} and multiplying both sides by
\[
(\qq{N}-\qq{X_{r}})\prod^{r-1}_{i=1}(\qq{N}-\qq{X_{i}+Y_{i+1}}),
\]
we get
\begin{equation}
L_{M-1}(\bX)=A_{M}(Y_{1})L_{M}(\bX)+B_{M}(X_1,Y_2)L_{M}(\bX_{\setminus\{1,2\}}).\label{eq:L_recur}
\end{equation}

Next, we show
\begin{equation}
R_{M-1}(\bX)=A_{M}(Y_{1})R_{M}(\bX)+B_{M}(X_{1},Y_{2})R_{M}(\bX_{\setminus\{1,2\}}).\label{eq:R_recur}
\end{equation}
Multiplying both sides of \eqref{eq:G_diff} by
\[
\frac{\qq{N-M}}{\qq{N-M}-\qq{Y_{1}}}\prod^{r}_{i=1}(\qq{N}-\qq{X_{i}+Y_{i}}),
\]
we have
\begin{equation}
\widetilde{G}^{\epsilon}_{q,M-1,N}(\bX)=A_{M}(\bX^{[1]})\widetilde{G}^{\epsilon}_{q,M,N}(\bX)+B_{M}(\bX^{[2]},\bX^{[1]})\widetilde{G}^{\epsilon}_{q,M,N}(\bX_{\setminus\{1,2\}})\label{eq:G_tilde_recur}
\end{equation}
where $\bX^{[1]}$ (resp.\ $\bX^{[2]}$) denotes the first (resp.\ second)
component of $\bX$. The case $r=1$ of \eqref{eq:R_recur} follows
from
\begin{align*}
R_{M-1}(Y_{1},X_{1}) & =A_{M}(Y_{1})R_{M}(Y_{1},X_{1})-\frac{q^{N}}{\qq{N}^{\epsilon}}A_{M}(Y_{1})+B_{M}(X_{1},Y_{1})+\frac{q^{N}}{\qq{N}^{\epsilon}}\qquad(\text{by }\eqref{eq:G_tilde_recur})\\
 & =A_{M}(Y_{1})R_{M}(Y_{1},X_{1})+B_{M}(X_{1},0)R_{M}(\emptyset)\qquad(\text{by }\eqref{eq:B_diff}).
\end{align*}
Assume that $r\geq2$. Put $c(T):=(-1)^{\mathrm{eo}(T)}\left(\frac{q^{N}}{\qq{N}^{\epsilon}}\right)^{\kappa(T)}$.
Then, by \eqref{eq:G_tilde_recur},
\begin{align*}
R_{M-1}(\bX) & =\sum_{T\in\mathcal{T}(r)}c(T)\widetilde{G}^{\epsilon}_{q,M-1,N}(\bX_{\setminus T})\\
 & =\sum_{T\in\mathcal{T}(r)}c(T)A_{M}((\bX_{\setminus T})^{[1]})\widetilde{G}^{\epsilon}_{q,M,N}(\bX_{\setminus T})\\
 & \quad+\sum_{T\in\mathcal{T}(r)}c(T)B_{M}((\bX_{\setminus T})^{[2]},(\bX_{\setminus T})^{[1]})\widetilde{G}^{\epsilon}_{q,M,N}((\bX_{\setminus T})_{\setminus\{1,2\}})\\
 & =\sum_{S}f(S)\widetilde{G}^{\epsilon}_{q,M,N}(\bX_{\setminus S})
\end{align*}
where
\[
f(S):=\sum_{\substack{T\in\mathcal{T}(r)\\
T=S
}
}c(T)A_{M}((\bX_{\setminus T})^{[1]})+\sum_{\substack{T\in\mathcal{T}(r)\\
S=T\cup\{m^{(1)}_{T},m^{(2)}_{T}\}
}
}c(T)B_{M}(\bX^{[m^{(2)}_{T}]},\bX^{[m^{(1)}_{T}]})
\]
where $m^{(1)}_{T}$ (resp.\ $m^{(2)}_{T}$) is the smallest (resp.\
second smallest) element of $[2r]\setminus T$. For $S\subset[2r]$,
we put
\[
\eta(S):=\max(\{\,i\,\mid\,[2i]\subset S\}).
\]
Then $f(S)=0$ except in the following cases: either $\eta(S)=0$ and $S\in\mathcal{T}(r)$,
or $\eta(S)\geq1$ and $S\setminus\{1,2\eta(S)\}\in\mathcal{T}(r)$.
For such cases, $f(S)$ is calculated as follows.
If $\eta(S)=0$, then
\[
f(S)=c(S)A_{M}(Y_{1}).
\]
If $\eta(S)=1$, then
\begin{align*}
f(S) & =c(S)A_{M}(Y_{2})+c(S\setminus\{1,2\})B_{M}(X_{1},Y_{1})\\
 & =c(S)A_{M}(Y_{1})+c(S\setminus\{1,2\})B_{M}(X_{1},Y_{2})\qquad(\text{by }\eqref{eq:B_diff}).
\end{align*}
If $\eta(S)=2$, then
\begin{align*}
f(S) & =c(S\setminus\{1,2\})B_{M}(X_{1},Y_{1})+c(S\setminus\{1,4\})B_{M}(X_{2},Y_{1})+c(S\setminus\{3,4\})B_{M}(X_{2},Y_{2})\\
 & =c(S\setminus\{1,2\})B_{M}(X_{1},Y_{2})\qquad(\text{by }\eqref{eq:B_diff}).
\end{align*}
If $\eta(S)\geq 3$, then, 
\begin{align*}
f(S) & =c(S\setminus\{1,2\eta(S)\})B_{M}(X_{\eta(S)},Y_{1})+c(S\setminus\{1,2\eta(S)-2\})B_{M}(X_{\eta(S)-1},Y_{1})\\
 & \quad+c(S\setminus\{3,2\eta(S)\})B_{M}(X_{\eta(S)},Y_{2})+c(S\setminus\{3,2\eta(S)-2\})B_{M}(X_{\eta(S)-1},Y_{2})\\
 & =0\qquad(\text{by }\eqref{eq:B_diff}).
\end{align*}
Thus,
\begin{align*}
\sum_{\eta(S)=0}f(S)\widetilde{G}^{\epsilon}_{q,M,N}(\bX_{\setminus S}) & =A_{M}(Y_{1})\sum_{\substack{T\in\mathcal{T}(r)\\
\{1,2\}\not\subset T
}
}c(T)\widetilde{G}^{\epsilon}_{q,M,N}(\bX_{\setminus T}),\\
\sum_{\eta(S)=1}f(S)\widetilde{G}^{\epsilon}_{q,M,N}(\bX_{\setminus S}) & =A_{M}(Y_{1})\sum_{\substack{T\in\mathcal{T}(r)\\
\{1,2\}\subset T
}
}c(T)\widetilde{G}^{\epsilon}_{q,M,N}(\bX_{\setminus T})\\
& \quad +B_{M}(X_{1},Y_{2})\sum_{\substack{T\in\mathcal{T}(r-1)\\
\{1,2\}\not\subset T
}
}c(T)\widetilde{G}^{\epsilon}_{q,M,N}((\bX_{\setminus\{1,2\}})_{\setminus T}),\\
\sum_{\eta(S)=2}f(S)\widetilde{G}^{\epsilon}_{q,M,N}(\bX_{\setminus S}) & =B_{M}(X_{1},Y_{2})\sum_{\substack{T\in\mathcal{T}(r-1)\\
\{1,2\}\subset T
}
}c(T)\widetilde{G}^{\epsilon}_{q,M,N}((\bX_{\setminus\{1,2\}})_{\setminus T}),\\
\sum_{\eta(S)\geq3}f(S)\widetilde{G}^{\epsilon}_{q,M,N}(\bX_{\setminus S}) & =0.
\end{align*}
Thus,
\begin{align*}
R_{M-1}(\bX) & =A_{M}(Y_{1})\sum_{T\in\mathcal{T}(r)}c(T)\widetilde{G}^{\epsilon}_{q,M,N}(\bX_{\setminus T})+B_{M}(X_{1},Y_{2})\sum_{T\in\mathcal{T}(r-1)}c(T)\widetilde{G}^{\epsilon}_{q,M,N}((\bX_{\setminus\{1,2\}})_{\setminus T})\\
 & =A_{M}(Y_{1})R_{M}(\bX)+B_{M}(X_{1},Y_{2})R_{M}(\bX_{\setminus\{1,2\}}),
\end{align*}
which completes the proof of \eqref{eq:R_recur}.

Now \Cref{thm:recur_for_gen_withM} follows from induction on $N-M$ by
\eqref{eq:L_recur}, \eqref{eq:R_recur}, and an easy base case $L_{N-1}(\bX)=R_{N-1}(\bX)$.
\Cref{cor:recur_for_gen} is a special case $M=0$ of \Cref{thm:recur_for_gen_withM}.
\end{proof}

\subsection{Explicit formulas and a proof of the main theorem}
We can now define the words that encode the explicit expansions of the truncated $q$-MZVs.
\begin{definition}\label{def:Eq}
For $\boldsymbol{c}=(c_{1},\dots,c_{2r})\in\mathbb{Z}_{\ge1}^{2r}$,
define $\mathfrak{E}_{q}^{0}(\boldsymbol{c})\in\h^{1}$ and $\mathfrak{E}_{q}^{1}(\boldsymbol{c})\in\h^{\geq2}$
recursively by $\mathfrak{E}_{q}^{0}(\emptyset)=\mathfrak{E}_{q}^{1}(\emptyset)=1$
and, for $r\ge1$ and $\epsilon\in\{0,1\}$, 
\begin{align*}
\mathfrak{E}_{q}^{\epsilon}(\boldsymbol{c}) & =-\sum_{\emptyset\neq B\subset[2r]_{\boldsymbol{c}}^{>1}}(-1)^{\#B}\mathfrak{E}_{q}^{\epsilon}(\boldsymbol{c}_{\emptyset,B})x^{\alpha(B)}\\
 & \quad+\sum_{\emptyset\neq B\subset[2r]_{\boldsymbol{c}}^{>1}}(-1)^{\#B}\mathfrak{E}_{q}^{\epsilon}(\boldsymbol{c}_{\emptyset,B})\left(\sum_{h=\alpha(B)+1}^{\beta(B)}yx^{h-1}\right)\\
 & \quad+\sum_{\substack{\emptyset\neq A\subset[2r]_{\boldsymbol{c}}^{1},\ A\in\mathcal{T}(r)\\
B\subset[2r]_{\boldsymbol{c}}^{>1}
}
}\sum_{h=1}^{\kappa(A)}(-1)^{\mathrm{eo}(A)+\#B+\kappa(A)-h}\binom{\kappa(A)-1}{h-1}\mathfrak{E}_{q}^{\epsilon}(\boldsymbol{c}_{A,B})yx^{h+\epsilon\kappa(A)+\beta_{A}(B)-1}.
\end{align*}
When $\alpha(B)>\beta(B)$, we interpret 
\[
\sum_{h=\alpha(B)+1}^{\beta(B)}yx^{h-1}:=-\sum_{h=\beta(B)+1}^{\alpha(B)}yx^{h-1}.
\]
Furthermore, define $\mathfrak{E}_{q}(\boldsymbol{c})\in\h^{1}$ and
$\mathfrak{D}_{q}(\boldsymbol{c})\in\h^{\geq2}$ by
\[
\mathfrak{E}_{q}(\boldsymbol{c}):=\mathfrak{E}_{q}^{0}(\boldsymbol{c}),\qquad\mathfrak{D}_{q}(\boldsymbol{c}):=(-1)^{c_{1}+\cdots+c_{2r}}\theta(\mathfrak{E}_{q}^{1}(\boldsymbol{c}))
\]
where $\theta:\h\to\h$ is the $\mathbb{Z}$-algebra automorphism defined by $\theta(x)=-x$, $\theta(y)=-y$. 
\end{definition}

By definition, the following theorem is equivalent to \cref{thm:main_finite} in the introduction.

\begin{theorem}
For $N\ge 1$, $r\ge0$, $\epsilon\in\{0,1\}$, and $\boldsymbol{k}=(k_{1},\dots,k_{r})$, $\boldsymbol{l}=(l_{1},\dots,l_{r})\in\mathbb{Z}_{\ge1}^{r}$,
we have 
\[
\xi_{q,N}^{\epsilon}(l_1,k_1,\dots,l_r,k_r)=Z_{q,N}^{\dagger}\bigl(\mathfrak{E}_{q}^{\epsilon}(l_{1},k_{1},\dots,l_{r},k_{r})\bigr).
\]
\end{theorem}
\begin{proof}
Divide both sides of \cref{cor:recur_for_gen} by $\qq{N}^{r}$ and
write $(W_{1},\dots,W_{2r})=(Y_{1},X_{1},\dots,Y_{r},X_{r})$. Then
\begin{align*}
 & \left(1-\frac{\qq{W_{1}}}{\qq{N}}\right)\left(1-\frac{\qq{W_{2r}}}{\qq{N}}\right)\left(\prod_{i=1}^{r-1}\left(1-\frac{\qq{W_{2i}}}{\qq{N}}-\frac{\qq{W_{2i+1}}}{\qq{N}}+\frac{\qq{W_{2i}}\qq{W_{2i+1}}}{\qq{N}}\right)\right)G_{q,N+1}^{\epsilon}(W_{1},\dots,W_{2r})\\
 & \qquad=\sum_{A\in\mathcal{T}(r)}\frac{(-1)^{\mathrm{eo}(A)}q^{N\kappa(A)}}{\qq{N}^{(1+\epsilon)\kappa(A)}}\frac{1}{\qq{N}^{r-\kappa(A)}}\widetilde{G}_{q,N}^{\epsilon}((W_{1},\dots,W_{2r})_{\setminus A}).
\end{align*}
For $\boldsymbol{c}=(c_1,\dots,c_{2r})\in\mathbb{Z}_{\geq 1}^{2r}$, we prove the claim
\[
\xi_{q,N}^{\epsilon}(\boldsymbol{c})=Z_{q,N}^{\dagger}(\mathfrak{E}_{q}^{\epsilon}(\boldsymbol{c}))
\]
by induction on $c_{1}+\cdots+c_{2r}$. The case $r=0$ is immediate,
so assume $r\ge1$.

The coefficient of $\qq{W_{1}}^{c_{1}-1}\cdots\qq{W_{2r}}^{c_{2r}-1}$ on
the left-hand side is 
\[
\sum_{B\subset[2r]_{\boldsymbol{c}}^{>1}}\frac{(-1)^{\#B}}{\qq{N}^{\alpha(B)}}\xi_{q,N+1}^{\epsilon}(\boldsymbol{c}_{\emptyset,B})=\xi_{q,N+1}^{\epsilon}(\boldsymbol{c})+\sum_{\emptyset\neq B\subset[2r]_{\boldsymbol{c}}^{>1}}\frac{(-1)^{\#B}}{\qq{N}^{\alpha(B)}}\xi_{q,N+1}^{\epsilon}(\boldsymbol{c}_{\emptyset,B}),
\]
while the coefficient on the right-hand side is 
\begin{align*}
 & \sum_{\substack{A\subset[2r]_{\boldsymbol{c}}^{1},\ A\in\mathcal{T}(r)\\
B\subset[2r]_{\boldsymbol{c}}^{>1}
}
}(-1)^{\mathrm{eo}(A)+\#B}\frac{q^{N\kappa(A)}}{\qq{N}^{(1+\epsilon)\kappa(A)+\beta_{A}(B)}}\xi_{q,N}^{\epsilon}(\boldsymbol{c}_{A,B})\\
 & \quad=\xi_{q,N}^{\epsilon}(\boldsymbol{c})+\sum_{\emptyset\neq B\subset[2r]_{\boldsymbol{c}}^{>1}}(-1)^{\#B}\frac{1}{\qq{N}^{\beta(B)}}\xi_{q,N}^{\epsilon}(\boldsymbol{c}_{\emptyset,B})\\
 & \qquad+\sum_{\substack{\emptyset\neq A\subset[2r]_{\boldsymbol{c}}^{1},\ A\in\mathcal{T}(r)\\
B\subset[2r]_{\boldsymbol{c}}^{>1}
}
}(-1)^{\mathrm{eo}(A)+\#B}\frac{q^{N\kappa(A)}}{\qq{N}^{(1+\epsilon)\kappa(A)+\beta_{A}(B)}}\xi_{q,N}^{\epsilon}(\boldsymbol{c}_{A,B}).
\end{align*}
Therefore 
\begin{align*}
 & \xi_{q,N+1}^{\epsilon}(\boldsymbol{c})-\xi_{q,N}^{\epsilon}(\boldsymbol{c})\\
 & =-\sum_{\emptyset\neq B\subset[2r]_{\boldsymbol{c}}^{>1}}\frac{(-1)^{\#B}}{\qq{N}^{\alpha(B)}}\bigl(\xi_{q,N+1}^{\epsilon}(\boldsymbol{c}_{\emptyset,B})-\xi_{q,N}^{\epsilon}(\boldsymbol{c}_{\emptyset,B})\bigr)\\
 & \quad+\sum_{\emptyset\neq B\subset[2r]_{\boldsymbol{c}}^{>1}}(-1)^{\#B}\left(\frac{1}{\qq{N}^{\beta(B)}}-\frac{1}{\qq{N}^{\alpha(B)}}\right)\xi_{q,N}^{\epsilon}(\boldsymbol{c}_{\emptyset,B})\\
 & \quad+\sum_{\substack{\emptyset\neq A\subset[2r]_{\boldsymbol{c}}^{1},\ A\in\mathcal{T}(r)\\
B\subset[2r]_{\boldsymbol{c}}^{>1}
}
}(-1)^{\mathrm{eo}(A)+\#B}\frac{q^{N\kappa(A)}}{\qq{N}^{(1+\epsilon)\kappa(A)+\beta_{A}(B)}}\xi_{q,N}^{\epsilon}(\boldsymbol{c}_{A,B}).
\end{align*}
Since 
\[
\frac{1}{\qq{N}^{m}}=1+q^{N}\sum_{h=1}^{m}\frac{1}{\qq{N}^{h}},
\]
we have
\begin{align*}
& 
\sum_{\emptyset\neq B\subset[2r]_{\boldsymbol{c}}^{>1}}(-1)^{\#B}\left(\frac{1}{\qq{N}^{\beta(B)}}-\frac{1}{\qq{N}^{\alpha(B)}}\right)\xi_{q,N}^{\epsilon}(\boldsymbol{c}_{\emptyset,B})\\
& =\sum_{\emptyset\neq B\subset[2r]_{\boldsymbol{c}}^{>1}}(-1)^{\#B}\left(\sum_{h=1}^{\beta(B)}\frac{q^{N}}{\qq{N}^{h}}-\sum_{h=1}^{\alpha(B)}\frac{q^{N}}{\qq{N}^{h}}\right)\xi_{q,N}^{\epsilon}(\boldsymbol{c}_{\emptyset,B}).
\end{align*}
Also, since 
\[
\frac{q^{Nm}}{\qq{N}^{(1+\epsilon)m}}=\sum_{m'=1}^{m}(-1)^{m-m'}\binom{m-1}{m'-1}\frac{q^{N}}{\qq{N}^{m'+\epsilon m}},
\]
we obtain 
\begin{align*}
 & \sum_{\substack{\emptyset\neq A\subset[2r]_{\boldsymbol{c}}^{1},\ A\in\mathcal{T}(r)\\
B\subset[2r]_{\boldsymbol{c}}^{>1}
}
}(-1)^{\mathrm{eo}(A)+\#B}\frac{q^{N\kappa(A)}}{\qq{N}^{(1+\epsilon)\kappa(A)+\beta_{A}(B)}}\xi_{q,N}^{\epsilon}(\boldsymbol{c}_{A,B})\\
 & \qquad=\sum_{\substack{\emptyset\neq A\subset[2r]_{\boldsymbol{c}}^{1},\ A\in\mathcal{T}(r)\\
B\subset[2r]_{\boldsymbol{c}}^{>1}
}
}\sum_{h=1}^{\kappa(A)}(-1)^{\mathrm{eo}(A)+\#B+\kappa(A)-h}\binom{\kappa(A)-1}{h-1}\frac{q^{N}}{\qq{N}^{h+\epsilon\kappa(A)+\beta_{A}(B)}}\xi_{q,N}^{\epsilon}(\boldsymbol{c}_{A,B}).
\end{align*}
Hence, by the induction hypothesis, 
\begin{align*}
 & \xi_{q,N+1}^{\epsilon}(\boldsymbol{c})-\xi_{q,N}^{\epsilon}(\boldsymbol{c})\\
 & =-\sum_{\emptyset\neq B\subset[2r]_{\boldsymbol{c}}^{>1}}\frac{(-1)^{\#B}}{\qq{N}^{\alpha(B)}}\Delta_{N}Z_{q,N}^{\dagger}(\mathfrak{E}_{q}^{\epsilon}(\boldsymbol{c}_{\emptyset,B}))\\
 & \quad+\sum_{\emptyset\neq B\subset[2r]_{\boldsymbol{c}}^{>1}}(-1)^{\#B}\left(\sum_{h=1}^{\beta(B)}\frac{q^{N}}{\qq{N}^{h}}-\sum_{h=1}^{\alpha(B)}\frac{q^{N}}{\qq{N}^{h}}\right)Z_{q,N}^{\dagger}(\mathfrak{E}_{q}^{\epsilon}(\boldsymbol{c}_{\emptyset,B}))\\
 & \quad+\sum_{\substack{\emptyset\neq A\subset[2r]_{\boldsymbol{c}}^{1},\ A\in\mathcal{T}(r)\\
B\subset[2r]_{\boldsymbol{c}}^{>1}
}
}\sum_{h=1}^{\kappa(A)}(-1)^{\mathrm{eo}(A)+\#B+\kappa(A)-h}\binom{\kappa(A)-1}{h-1}\frac{q^{N}}{\qq{N}^{h+\epsilon\kappa(A)+\beta_{A}(B)}}Z_{q,N}^{\dagger}(\mathfrak{E}_{q}^{\epsilon}(\boldsymbol{c}_{A,B})),
\end{align*}
where 
\[
\Delta_{N}Z_{q,N}^{\dagger}(u):=Z_{q,N+1}^{\dagger}(u)-Z_{q,N}^{\dagger}(u).
\]
Using 
\[
\Delta_{N}Z_{q,N}^{\dagger}(uyx^{k-1})=\frac{q^{N}}{\qq{N}^{k}}Z_{q,N}^{\dagger}(u)\qquad\text{and}\qquad\Delta_{N}Z_{q,N}^{\dagger}(ux^{k})=\frac{1}{\qq{N}^{k}}\Delta_{N}Z_{q,N}^{\dagger}(u),
\]
we may rewrite the previous identity as 
\begin{align*}
& \xi_{q,N+1}^{\epsilon}(\boldsymbol{c})-\xi_{q,N}^{\epsilon}(\boldsymbol{c}) \\
& =-\sum_{\emptyset\neq B\subset[2r]_{\boldsymbol{c}}^{>1}}(-1)^{\#B}\Delta_{N}Z_{q,N}^{\dagger}(\mathfrak{E}_{q}^{\epsilon}(\boldsymbol{c}_{\emptyset,B})x^{\alpha(B)})\\
 & \quad+\sum_{\emptyset\neq B\subset[2r]_{\boldsymbol{c}}^{>1}}(-1)^{\#B}\left(\sum_{h=1}^{\beta(B)}\Delta_{N}Z_{q,N}^{\dagger}(\mathfrak{E}_{q}^{\epsilon}(\boldsymbol{c}_{\emptyset,B})yx^{h-1})-\sum_{h=1}^{\alpha(B)}\Delta_{N}Z_{q,N}^{\dagger}(\mathfrak{E}_{q}^{\epsilon}(\boldsymbol{c}_{\emptyset,B})yx^{h-1})\right)\\
 & \quad+\sum_{\substack{\emptyset\neq A\subset[2r]_{\boldsymbol{c}}^{1},\ A\in\mathcal{T}(r)\\
B\subset[2r]_{\boldsymbol{c}}^{>1}
}
}\sum_{h=1}^{\kappa(A)}(-1)^{\mathrm{eo}(A)+\#B+\kappa(A)-h}\binom{\kappa(A)-1}{h-1}\Delta_{N}Z_{q,N}^{\dagger}(\mathfrak{E}_{q}^{\epsilon}(\boldsymbol{c}_{A,B})yx^{h+\epsilon\kappa(A)+\beta_{A}(B)-1})\\
 & =\Delta_{N}Z_{q,N}^{\dagger}(\mathfrak{E}_{q}^{\epsilon}(\boldsymbol{c})).
\end{align*}
Thus $\xi_{q,N}^{\epsilon}(\boldsymbol{c})$ and $Z_{q,N}^{\dagger}(\mathfrak{E}_{q}^{\epsilon}(\boldsymbol{c}))$
have the same forward difference with respect to $N$. Since both vanish at $N=1$, they are equal for all $N$. 
\end{proof}

By taking the limit as $q\to 1$, we obtain the classical case of \cref{thm:main_finite}.
In this limit, all terms except the highest-weight part vanish, and
hence the resulting corollary is as follows.

\begin{corollary}
For $\boldsymbol{c}=(c_{1},\dots,c_{2r})\in\mathbb{Z}_{\ge1}^{2r}$,
define $\mathfrak{E}^{0}(\boldsymbol{c})\in\h^{1}$ and $\mathfrak{E}^{1}(\boldsymbol{c})\in\h^{\geq2}$
recursively by $\mathfrak{E}^{0}(\emptyset)=\mathfrak{E}^{1}(\emptyset)=1$
and, for $r\ge1$ and $\epsilon\in\{0,1\}$, 
\begin{align*}
\mathfrak{E}^{\epsilon}(\boldsymbol{c}) & =-\sum_{\substack{ B\subset[2r]_{\boldsymbol{c}}^{>1},\,\mathrm{eo}(B)=0\\
B\neq \emptyset
}
}(-1)^{\#B}\mathfrak{E}^{\epsilon}(\boldsymbol{c}_{\emptyset,B})x^{\#B}\\
 & \quad-\sum_{\substack{B\subset[2r]_{\boldsymbol{c}}^{>1},\,\mathrm{eo}(B)=0\\
\#B\geq2
}
}(-1)^{\#B}\mathfrak{E}^{\epsilon}(\boldsymbol{c}_{\emptyset,B})yx^{\#B-1}\\
 & \quad+\sum_{\substack{A\subset[2r]_{\boldsymbol{c}}^{1},\,A\in\mathcal{T}(r)\\
B\subset[2r]_{\boldsymbol{c}}^{>1},\,\mathrm{oe}(\sigma_A(B))=0\\
\#A+\#B\geq 2
}
}(-1)^{\mathrm{eo}(A)+\#B}\mathfrak{E}^{\epsilon}(\boldsymbol{c}_{A,B})yx^{(1+\epsilon)\kappa(A)+\#B-1}.
\end{align*}
Furthermore, define $\mathfrak{E}(\boldsymbol{c})\in\h^{1}$ and $\mathfrak{D}(\boldsymbol{c})\in\h^{\geq2}$
by 
\[
\mathfrak{E}(\boldsymbol{c}):=\mathfrak{E}^{0}(\boldsymbol{c}),\qquad\mathfrak{D}(\boldsymbol{c}):=\mathfrak{E}^{1}(\boldsymbol{c}).
\]
For $N\ge 1$, $r\ge0$ and $\boldsymbol{l},\boldsymbol{k}\in\mathbb{Z}_{\ge1}^{r}$,
define
\begin{align*}
\zeta_{N}\!\begin{pmatrix}\boldsymbol{l}\\
\boldsymbol{k}
\end{pmatrix} & :=\sum_{\substack{0<n_{1,1}\le\cdots\le n_{1,l_{1}}\\
<n_{2,1}\le\cdots\le n_{2,l_{2}}\\
\cdots\\
<n_{r,1}\le\cdots\le n_{r,l_{r}}<N
}
}\prod_{j=1}^{r}\frac{1}{(N-n_{j,1})\cdots(N-n_{j,l_{j}-1})n_{j,l_{j}}^{k_{j}}},\\
\zeta_{N}(\boldsymbol{k}) & :=\sum_{0<n_{1}<\cdots<n_{r}<N}\prod_{j=1}^{r}\frac{1}{n_{j}^{k_{j}}},\\
\zeta_{N}^{\ds}(\boldsymbol{k}) & :=\sum_{A\subset\{i\mid k_{i}=1\}}\sum_{\substack{0<n_{1}\le\cdots\le n_{r}<N\\
n_{i}<n_{i+1}\quad(i\notin A)
}
}\left(\prod_{i\in A}\frac{1}{N-n_{i}}\right)\left(\prod_{i\notin A}\frac{1}{n_{i}^{k_{i}}}\right).
\end{align*}
Define $Z_{N}:\h^{1}\to\mathbb{Q}$ by
\[
Z_{N}(yx^{k_{1}-1}\cdots yx^{k_{r}-1}):=\zeta_{N}(k_{1},\dots,k_{r}).
\]
Then, for $N\ge 1$, $r\ge0$ and $\boldsymbol{k}=(k_{1},\dots,k_{r})$, $\boldsymbol{l}=(l_{1},\dots,l_{r})\in\mathbb{Z}_{\ge1}^{r}$,
we have 
\begin{align*}
\zeta_{N}\!\begin{pmatrix}\boldsymbol{l}\\
\boldsymbol{k}
\end{pmatrix} & =Z_{N}\bigl(\mathfrak{E}(l_{1},k_{1},\dots,l_{r},k_{r})\bigr),\\
\zeta_{N}^{\ds}(\underbrace{1,\dots,1}_{l_{1}-1},k_{1}+1,\dots,\underbrace{1,\dots,1}_{l_{r}-1},k_{r}+1) & =Z_{N}\bigl(\mathfrak{D}(l_{1},k_{1},\dots,l_{r},k_{r})\bigr).
\end{align*}
\end{corollary}

\end{document}